\input ./style/arxiv-vmsta.cfg
\documentclass[numbers,compress,v1.0.1]{vmsta}
\usepackage{vtexbibtags}

\usepackage{authorquery}

\volume{5}
\issue{1}
\pubyear{2018}
\firstpage{113}
\lastpage{127}
\aid{VMSTA98}
\doi{10.15559/18-VMSTA98}
\articletype{research-article}

\numberwithin{equation}{section}
\startlocaldefs
\def\@journal@url{http://www.vmsta.org}
\def\@credit{%
  \vbox to 0pt{%
    \vskip-1.85pc
    \hskip-\textwidth
    \noindent
    \raise4mm\hbox to \textwidth{%
      \footnotesize
      \href{\@journal@url}{www.vmsta.org}%
      \hfill
      \href{\@vtex@url}{\includevtexlogo}%
      }%
    }%
  }
\newcommand{\rrVert}{\Vert}
\newcommand{\llVert}{\Vert}
\newcommand{\rrvert}{\vert}
\newcommand{\llvert}{\vert}
\allowdisplaybreaks

\def\R{\mathbb{R}}
\def\E{\mathbb{E}}
\def\N{\mathbb{N}}
\def\Nb{\overline\N}

\newtheorem{thm}{Theorem}
\newtheorem{lm}[thm]{Lemma}

\theoremstyle{definition}
\newtheorem{defn}{Definition}

\theoremstyle{remark}
\newtheorem*{notation}{Notation}
\ifdefined\HCode 
\def\mytag{\qquad\qquad\qquad\qquad\qquad\qquad\mbox{(Lemma 3)}\nonumber}
\else 
\def\mytag{\tag{\small Lemma~\ref{Cheng}}}
\fi

\endlocaldefs

\begin{aqf}
\querytext{Q1}{Isn't it Lemma 3?}
\querytext{Q2}{Looks like $f\in$ was superfluous here.}
\end{aqf}
\begin{document}

\begin{frontmatter}
\pretitle{Research Article}

\title{On backward Kolmogorov equation related to CIR process}

\author{\inits{V.}\fnms{Vigirdas}~\snm{Mackevi\v cius}\thanksref{cor1}\ead[label=e1]{vigirdas.mackevicius@mif.vu.lt}}
\author{\inits{G.}\fnms{Gabriel\.e}~\snm{Mongirdait\.e}}
\thankstext[type=corresp,id=cor1]{Corresponding author.}
\address{Institute of Mathematics, Faculty of Mathematics and Informatics,\break \institution{Vilnius University},
Naugarduko 24, Vilnius,~\cny{Lithuania}}



\markboth{V. Mackevi\v cius, G. Mongirdait\.e}{On backward Kolmogorov equation related to CIR process}

\begin{abstract}
We consider the existence of a classical smooth solution to the backward
Kolmogorov equation
\begin{align*}
\begin{cases}
\partial_t u(t,x)=Au(t,x),& x\ge0,\ t\in[0,T],\\
u(0,x)=f(x),& x\ge0,
\end{cases} %
\end{align*}
where $A$ is the generator of the CIR process, the solution to the
stochastic differential equation
\begin{equation*}
X^x_t=x+\int_0^t\theta
\bigl(\kappa-X^x_s\bigr)\,ds+\sigma\int
_0^t\sqrt {X^x_s}
\,dB_s, \quad x\ge0,\ t\in[0,T],
\end{equation*}
that is, $Af(x)=\theta(\kappa-x)f'(x)+\frac{1}{2}\sigma^2xf''(x)$, $
x\ge0$ ($\theta,\kappa,\sigma>0$). Alfonsi \cite{Alfonsi} showed that
the equation has a smooth solution with partial derivatives of polynomial
growth, provided that the initial function $f$ is smooth with derivatives
of polynomial growth. His proof was mainly based on the analytical formula
for the transition density of the CIR process in the form of a~rather
complicated function series. In this paper, for a CIR process satisfying
the condition $\sigma^2\le4\theta\kappa$, we present a direct proof based
on the representation of a CIR process in terms of a~squared Bessel process
and its additivity property.
\end{abstract}
%

\received{\sday{20} \smonth{11} \syear{2017}}
\revised{\sday{31} \smonth{1} \syear{2018}}
\accepted{\sday{2} \smonth{2} \syear{2018}}
\publishedonline{\sday{6} \smonth{3} \syear{2018}}
\end{frontmatter}

\section{Introduction}
Let us recall the well-known relationship between the one-dimensional
stochastic differential equation
(SDE)
\begin{equation}
\label{eq:general SDE} X^x_t=x+\int_0^tb
\bigl(X^x_s\bigr)\,dt+\int_0^t
\sigma\bigl(X^x_s\bigr)\,dB_s,\quad
X^x_0=x,
\end{equation}
and the following parabolic partial differential equation (PDE), called
the backward Kolmogorov equation, with initial condition
\begin{align}
\begin{cases}\label{eq: Kolm. PDE}
\partial_tu(t,x)=Au(t,x),\\
u(0,x)=f(x),
\end{cases} %
\end{align}
where $Af=bf'+\frac{1}2\sigma^2f''$ is the generator of the diffusion
defined by SDE \eqref{eq:general SDE}. If the coefficients $b,\sigma:\R
\to\R$ and the initial function $f$ are sufficiently ``good,'' then
the function $u=u(t,x):=\mathbb{E}f(X^x_t)$ is a~(classical) solution
to PDE~\eqref{eq: Kolm. PDE}. From this by It\^o's formula it follows
that the random process
$$M^x_t:=u(T-t,X^x_t),\quad t\in[0,T],$$
is a martingale with mean $\E
\, M^x_t=f(x)$ satisfying the final condition $M^x_T=f(X^x_T)$. This
fact is essential in rigorous proofs of the convergence rates of weak
approximations of SDEs. The higher the convergence rate, the greater
smoothness of the coefficients, and the final condition is to be
assumed to get a sufficient smoothness of the solution~$u$ to~\eqref
{eq: Kolm. PDE}. The question of the existence of smooth classical
solutions to the backward Kolmogorov equation is more complicated than
it might seem from the first sight. General results typically require
smoothness and polynomial growth of several higher-order derivatives of
the coefficients; we refer to the book by Kloeden and Platen \cite
{Kloeden-Platen}, Theorem 4.8.6 on p.~153.

However, the coefficients of many SDEs used in financial mathematics
are not sufficiently good, and therefore the general theory is not
applicable. A classic example is the well-known Cox--Ingersoll--Ross
(CIR) process \cite{Cir}, the solution to the SDE
\begin{align}
\label{eq: CIR} X^x_t=x+\int_0^t
\theta\bigl(\kappa-X^x_s\bigr)\,ds+\int
_0^t\sigma\sqrt {X^x_s}
\,dB_s, \quad t\in[0,T],
\end{align}
with parameters $\theta, \kappa, \sigma>0$, $x\ge0$, where the
diffusion coefficient $\tilde\sigma(x)=\sigma\sqrt{x}$ has unbounded
derivatives.

Alfonsi \cite[Prop.~4.1]{Alfonsi}, using the known expression of the
transition density of CIR process by a~rather complicated function
series, gave an ad~hoc proof that, indeed, $u=u(t,x):=\mathbb
{E}f(X^x_t)$ is a classic solution to the PDE~\eqref{eq: Kolm. PDE},
where
$$
Af(x)=\theta(\kappa-x)f'(x)+\frac{1}2\sigma^2xf''(x),\quad x\ge0,
$$
is the generator of the CIR process~\eqref{eq: CIR}. Moreover, he
proved that if $f:\R_+\to\R$ is sufficiently smooth with partial
derivatives of polynomial growth, then so is the solution $u$.

In this paper, in case the coefficients of Eq.~\eqref{eq: CIR} satisfy
the condition $\sigma^2\le4\theta\kappa$, we give another proof of
this result, where we do not use the transition function. We believe
that our approach will be applicable to a wider class of
``square-root-type'' processes for which an explicit form of the
transition function is not known (e.g., the well-known square-root
stochastic-volatility Heston process \cite{Heston}). The main tools are
the additivity property of CIR processes and their representation in
terms of squared Bessel processes. More precisely, we use, after a
smooth time--space transformation, the expression of the solution
to Eq.~\eqref{eq: CIR} in the form $X^x_t=(\sqrt x+B_t)^2+Y_t$, where
$Y$ is a squared Bessel process independent from~$B$. The main
challenge is the negative powers of $x$ appearing in the expression
of $u(t,x)=\E f(X^x_t)$ after differentiation with respect to $x>0$.
To overcome it, we use a ``symmetrization'' trick (see Step~1 in the
proof of Theorem~\ref{pol1}) based on the simple fact that replacing $
B_t$ by the ``opposite'' Brownian motion $\bar B_t:=-B_t$ does not
change the distribution of~$X^x_t$.

Both proofs, Alfonsi's and ours, are ``probabilistic.'' It is
interesting whether there are similar results with ``nonprobabilistic''
proofs in the literature. Equation (\ref{eq: Kolm. PDE}) seems to be a very simple
equation, with coefficients analytic everywhere and the diffusion
nondegenerate everywhere except a single point. However, although there
is a vast literature on degenerate parabolic and elliptic
equations, we could find only a few related results, which, however, do
not include the case of initial functions $f$ from $C^n_{\mathit{pol}}(\mathbb
{R}_+)$ or $C^\infty_{\mathit{pol}}(\mathbb{R}_+)$ (see the notation in the
Introduction); instead, the boundedness of $f$ and its derivatives is
assumed as a rule. For example, general Theorem~1.1 of
Feehan and Pop~\cite{P.M.N. Feehan and C.A. Pop} (see also Cerrai~\cite
{Cerrai}) in our particular (one-dimensional) case gives an a~priori
estimate of the form
\[
\|u\|_{\mathcal{C}^{2+\alpha}([0,T]\times[0,\infty))}\le \xch{C\|f\| _{\mathcal{C}_p^{2+\alpha}([0,\infty))},}{C\|f\| _{\mathcal{C}_p^{2+\alpha}([0,\infty))}}
\]
in terms of the corresponding H\"older and weighted H\"older space
\textit{supremum} norms.
\def\R{\mathbb{R}}
\def\E{\mathbb{E}}
\def\ds{\,d s}
\def\dBs{\,d B_s}
\def\N{\mathbb{N}}

\section{Preliminaries}

\begin{defn}[\cite{CirBessel}, Def.~6.1.2.1]\label{def:BESQ}
For every $\delta{\geq} 0$ and $x {\geq} 0$, the unique strong
solution $Y$ to the equation
\begin{equation}
Y_t = x + \delta t + 2\int_0^t
\sqrt{Y_s}\dBs,\quad \xch{t {\geq} 0}{t {\geq} 0,}
\end{equation}
is called a squared Bessel process with dimension $\delta$, starting
at $x$ (BESQ$^\delta_x$ for short).
We further denote it by $Y^\delta_t(x)$ or $Y^\delta(t,x) $, and
also, $Y^\delta_t:=Y^\delta_t(0)$.
\end{defn}

\begin{lm}[See \cite{CirBessel}, Section~6.1]
\label{lm:BESQn}
Let $B=(B^1, B^2,\ldots,B^n)$ be a standard $n$-dimensional Brownian motion,
$n\in\mathbb{N}$. Then the process
\begin{align*}
R^2_t:= \llVert z+B_t \rrVert
^2=\sum_{i=1}^n
\bigl(z_i+B^i_t\bigr)^2,\quad t
\ge0,
\end{align*}
where $z=(z_1,\dots, z_n)\in\R^n$, coincides in distribution with $
Y^n_t(\|z\|), $ that is, with a BESQ$^n_x$ random process starting at
$x=\|z\|=\sqrt{\sum_{i=1}^nz^2_i}$. In particular,
\begin{equation}
\label{eq: } Y^n_t(x)\stackrel{d} {=}\bigl(\sqrt
x+B^1_t\bigr)^2+\sum
_{i=2}^n\bigl(B^i_t
\bigr)^2 \stackrel{d} {=}(\sqrt x+\xi\sqrt t)^2+Y^{n-1}_t
,\quad t\ge0,
\end{equation}
where $\xi$ is a standard normal variable independent of $Y^{n-1}_t$, and $\stackrel{d}{=}$ means equality in distribution.
\end{lm}

\begin{lm}[\cite{CirBessel}, Prop.~6.3.1.1]\label{lm:CIR-BESQ}
The distribution of CIR process \eqref{eq: CIR} can be expressed in
terms of a squared Bessel process as follows:
\begin{equation}
\label{eq: nCIR in terms of BESQ} X_t(x)\stackrel{d} {=}\mathrm{e}^{-\theta t}Y^\delta
\biggl(\frac{\sigma
^2}{4\theta}\bigl(\mathrm{e}^{\theta t}-1\bigr),x \biggr),\quad t
\ge0,
\end{equation}
where $\delta=4\theta\kappa/\sigma^2$.
\end{lm}

We will frequently use differentiation under the integral sign (in
particular, under the expectation sign). Without special mentioning,
this will be clearly justified by \querymark{Q1}Lemma~\ref{Cheng},
which seems to be a folklore theorem;
we refer to technical report \cite{SteveCheng}.
\begin{defn}
\label{localintegrability}
Let $(E,\mathcal{A},\mu)$ be a measure space. Let $X\subset\R^k$ be
an open set, and $f: X\times E\rightarrow\mathbb{R}$ be a~measurable
function. The function $f$ is said to be locally integrable in $X$ if
\begin{align*}
\int\limits
_K\int\limits
_E\bigl\llvert f(x,\omega)\bigr\rrvert\mu(
d\omega)dx<\infty
\end{align*}
for all compact sets $K\subset X$.
\end{defn}

\begin{lm}[Differentiation under the integral sign; see
\cite{SteveCheng}, Thm.~4.1]
\label{Cheng}
Let $(E,\mathcal{A},\mu)$, $X$, and let~$f$ be as in
Definition~\ref{localintegrability}. Suppose that
$f$ has partial derivatives $\frac{\partial f}{\partial x_i}(x,\omega)$ for all
$(x,\omega)\in X\times E$ and that both $f$ and $\frac{\partial
f}{\partial x_i}$ are locally integrable in $X$. Then
\begin{align*}
\frac{\partial}{\partial x_i}\int\limits
_E f(x,\omega)\mu(d\omega) =\int\limits
_E
\frac{\partial}{\partial x_i}f(x,\omega)\mu(d\omega)
\end{align*}
for almost all $x\in X$.
In particular, if both sides are continuous in $X$, then we have
equality for all $x\in X$.
\end{lm}

\begin{notation}
As usual, $\N$ and $\R$ are the sets of natural and real numbers,
$\R_+:=[0,\infty)$, and $\Nb:=\N\cup\{0\}$. We denote by \querymark
{Q2}$C^n_{\mathit{pol}}(\mathbb{R}_+)$ the set of $n$ times continuously
differentiable functions $f:\R_+\to\R$ such that there exist
constants $C_i\ge0$ and $k_i\in\mathbb{N}$, $i=0,1,\ldots,n$,
such that
\begin{equation}
\label{eq: polinom. f isv. augimas iki n} \bigl\llvert f^{(i)}(x)\bigr\rrvert \le C_i
\bigl(1+x^{k_i}\bigr),\quad  x\ge0,\
\end{equation}
for all $i=0,1,\dots,n.$ Then, following Alfonsi~\cite{Alfonsi2010},
we say that the set of constants \{$(C_i,k_i)$, $i=0,1,\ldots,n$\}
is good for~$f$.
If $f\in C^\infty_{\mathit{pol}}(\mathbb{R}_+)$, that is, $f:\R_+\to\R$
is infinitely differentiable and there exist constants $C_i\ge0$ and
$k_i\in\mathbb{N}$, $i\in\Nb$, such that
\begin{equation}
\label{eq: polinom. f VISU isv. augimas} \bigl\llvert f^{(i)}(x)\bigr\rrvert \le C_i
\bigl(1+x^{k_i}\bigr),\quad  x\ge0,\ i\in\xch{\overline\N,}{\overline\N.}
\end{equation}
then
the sequence of constants \{$(C_i,k_i)$, $i\in\Nb $\} is said to
be good for $f$. Finally, by $C\ge0$ and $k\in\N$ we will denote
constants that depend only on the good set of a function $f$ and may
very from line to line.
\end{notation}

\section{Existence and properties of a solution to backward Kolmogorov
equation related to CIR process}
Our main result is a direct proof of the following:
\begin{thm}[cf.~Alfonsi~\cite{Alfonsi}, Prop.~4.1]
\label{pol1}
Let $X_t(x)=X^x_t$ be a CIR process with coefficients satisfying the
condition $\sigma^2\le4\theta\kappa$ and starting at $x\ge0$. Let
$f\in C^q_{\mathit{pol}}(\mathbb{R}_+)$ for some $q\ge4$. Then the function
\[
u(t,x):=\mathbb{E}f\bigl(X_t(x)\bigr),\quad x\ge0,\ t\in[0,T],
\]
is $l$ times continuously differentiable in $x\ge0$ and $l'$
times continuously differentiable in $t\in[0,T]$ for $l,l'\in\N$
such that $2l+4l' \le q$. Moreover, there exist constants $C\ge0$ and
$k\in\mathbb{N}$, depending only on a good set $\{(C_i,k_i)$, $
i=0,1,\ldots,q\}$ for $f$, such that
%
\begin{equation}
\label{eq: nel. d/dx Ef(Xt(x)}
\bigl\llvert \partial^j_x \partial^i_tu(t,x)\bigr\rrvert \leq C\bigl(1+x^k \bigr), \quad x\ge0,\ t\in[0,T],\
\end{equation}
for $j=0,1,\ldots,l,\ i=0,1,\ldots,l'$. In particular, $u(t,x)$ is
a \emph{(}classical\emph{)} solution to the Kolmogorov backward
equation \eqref{eq: Kolm. PDE} for $(t,x)\in[0,T]\times\R_+$.

As a consequence, if $f\in C^\infty_{\mathit{pol}}(\mathbb{R}_+)$, then $u(t,x)$
is infinitely differentiable on $[0,T]\times\R_+$,
and estimate~\eqref{eq: nel. d/dx Ef(Xt(x)} holds for all $i,j\in\N
_+$ with $C$ and $k$ depending on $(i,j)$ and a good sequence $
\{(C_i,k_i)$, $i\in\Nb \}$ for~$f$.
\end{thm}

\begin{proof}
We first focus ourselves on the differentiability in $x\ge0$. By
Lemma~\ref{lm:CIR-BESQ} the process $X_t(x)$ can be reduced, by a
space--time transformation, to the BESQ$^\delta$ process $Y^\delta
_t(x)$ with $\delta=\frac{4\theta\kappa}{\sigma^2}\ge1$. Since
only bounded smooth functions of $t\in[0,T]$ are involved in \eqref
{eq: nCIR in terms of BESQ}, it suffices to show estimate~\eqref{eq:
nel. d/dx Ef(Xt(x)}
for $Y^\delta_t(x)$, $t\in[0,\tilde T],$ instead of $X_t(x)$, $
t\in[0,T]$, with $\tilde T=\frac{1}\theta\ln(1+\frac{4\theta T}{\sigma
^2})$. With an abuse of notation, we further write $T$ instead of $
\tilde T$. We proceed by induction on $l$.

\emph{Step 1}. Let $l=1$.
First, suppose that $\delta= n\in\N$. By Lemma~\ref{lm:BESQn} we
have
\begin{equation}
Y^n_t(x)\stackrel{d}=(\sqrt{x}+\xi\sqrt{t})^2+Y^{n-1}_t,
\end{equation}
where $\xi\sim\mathcal{N}(0,1)$ is independent of $Y^{n-1}_t$ (in
the case $n=1$, $Y^0_t:=0$).
Denote
\begin{align*}
Y^{+}_t(x)&:=(\sqrt{x}+\xi\sqrt{t})^2,\qquad
Y^{-}_t(x):=(\sqrt{x}-\xi \sqrt{t})^2.
\end{align*}

Since the distributions of $Y^{+}_t(x)$ and $Y^{-}_t(x)$ coincide,
we have
\begin{align}
&\partial_x\mathbb{E}f\bigl(Y^n_t(x)\bigr)
\notag
\\
&\quad= \partial_x\mathbb{E}f \bigl(Y^+_{t}(x)+Y^{n-1}_{t}
\bigr)
\nonumber
\\
&\quad=\frac{1}{2} \bigl[\partial_x\mathbb{E}f
\bigl(Y^+_{t}(x)+Y^{n-1}_{t} \bigr)+
\partial_x\mathbb{E}f \bigl(Y^-_{t}(x)+Y^{n-1}_{t}
\bigr) \bigr]
\nonumber
\\
&\quad=\frac{1}{2}\mathbb{E} \biggl[f' \bigl(Y^{+}_{t}(x)+Y^{n-1}_{t}
\bigr) \biggl(1+\xi\sqrt{\frac{t}{x}} \biggr)\mytag
\\
&\qquad +f' \bigl(Y^{-}_{t}(x)+Y^{n-1}_{t}
\bigr) \biggl(1-\xi\sqrt {\frac{t}{x}} \biggr) \biggr]
\nonumber
\\
&\quad=\mathbb{E}f' \bigl(Y^{n}_{t}(x)
\bigr)
\nonumber
\\
&\qquad +\frac{1}{2}\sqrt{\frac{t}{x}}\mathbb{E} \bigl\{\xi
\bigl[f' \bigl(Y^{+}_{t}(x)+Y^{n-1}_{t}
\bigr)-f' \bigl(Y^{-}_{t}(x)+Y^{n-1}_{t}
\bigr) \bigr] \bigr\}
\nonumber
\\
&\quad=\mathbb{E}f' \bigl(Y^{n}_{t}(x)
\bigr)+\frac{1}{2}\sqrt{t}\, \mathbb{E} \bigl(\xi g_1\bigl(x,
\xi\sqrt t, Y^{n-1}_{t}\bigr) \bigr)=:P(t,x)+R(t,x),\quad
x>0,\label{eq:P(t,x),R(t,x)}
\end{align}
where
\[
g_1(x,a,b):=\frac{f' ((\sqrt x+a)^2+b )-f' ((\sqrt x-a)^2+b
)}{\sqrt x},\quad x>0,\ a\in\R,\ b\ge0.
\]
We now estimate $P(t,x)$ and $R(t,x)$ separately.
By the well-known inequality
\begin{equation}
\label{eq: sum a_i} \Biggl\llvert \sum_{i=1}^na_i
\Biggr\rrvert ^p\le n^{p-1}\sum
_{i=1}^n\llvert a_i\rrvert^p
\quad \mbox{for any } n\in\N,\ p\ge1,\ a_i\in\R,\ i=1,2,\ldots,n,
\end{equation}
we have the following estimates:
\begin{align*}
\E\bigl(Y^{\pm}_t(x)\bigr)^p&=\E(\sqrt{x}\pm
\xi\sqrt{t})^{2p} \le2^{2p-1}\bigl(x^p+\E\llvert
\xi\rrvert^{2p}t^p\bigr)
\\
&=2^{2p-1} \biggl(x^p+\frac{2^p\varGamma(p+\frac{1}2)}{\sqrt\pi}t^p
\biggr)
\\
&\le C\bigl(1+x^p\bigr),\quad x\ge0,\ t\in[0,T],\\
\E\bigl(Y^{n}_t\bigr)^p &=\E \Biggl(\sum
_{i=1}^n \bigl\llvert B^i_t
\bigr\rrvert^{2} \Biggr)^p \le n^{p-1}\sum
_{i=1}^n\E \bigl\llvert
B^i_t\bigr\rrvert^{2p}
\\
&=n^{p}\frac{2^p\varGamma(p+\frac{1}2)}{\sqrt\pi}t^p \le C,\quad t\in[0,T],
\end{align*}
and, as a consequence,
\begin{align}
\E\bigl(Y^n_t(x)\bigr)^p&=\E
\bigl(Y^{+}_t(x)+Y^{n-1}_t
\bigr)^p \le2^{p-1}\E \bigl(\bigl(Y^{+}_t(x)
\bigr)^p+\E\bigl(Y^{n-1}_t\bigr)^p
\bigr)
\nonumber
\\
&\le C\bigl(1+x^p\bigr),\quad x\ge0,\ t\in[0,T].\label{eq:EY^p< C(1+x^p)}
\end{align}
Now, for $P(t,x)$, we have
\begin{align}
\bigl\llvert P(t,x)\bigr\rrvert &=\mathbb{E}\bigl\llvert f'
\bigl(Y^{n}_{t}(x) \bigr)\bigr\rrvert \leq C_1
\bigl(1+\E\bigl(Y^{n}_{t}(x)\bigr)^{k_1} \bigr)
\leq C_1 \bigl(1+C\bigl(1+x^{k_1}\bigr) \bigr)
\nonumber
\\
&\le C\bigl(1+x^{{k}_1}\bigr),\quad x\ge0,\ t\in[0,T],\label{eq: P(t,x)<...}
\end{align}
where the constant $C$ depends only on $C_1$, $k_1$, $T$, and $
n$.

At this point, we need the following technical lemma, which we will
prove in the Appendix.

\begin{lm}\label{lem: g(x,a,b)} For a function $f:\R_+\to\R$, define
the function
\[
g(x;a,b) :=\frac{f ((\sqrt x+a)^2+b )-f ((\sqrt x-a)^2+b )}{\sqrt
x}, \quad x>0,\ a\in\R,\ b\in\R_+.
\]
If $f\in C^{q}_{\mathit{pol}}(\mathbb{R}_+)$ for some $q=2l+1\in\N$ $(l\in
\Nb)$, then the function $g$ is extendable to a continuous function
on $\R_+\times\R\times\R_+ $ such that $g(\cdot;a,b)\in
C^l_{\mathit{pol}}(\R_+)$ for all $a\in\R$ and $b\in\R_+$. Moreover, there
exist constants $C\ge0$ and $k\in\mathbb{N}$, depending only on a
good set $\{(C_i,k_i)$, $i=0,1,\ldots,q\}$ for~$f$, such that
\begin{equation}
\bigl\llvert \partial^j_xg(x;a,b)\bigr\rrvert \le C
\llvert a\rrvert \bigl(1+x^k+\bigl\llvert a^2+b\bigr
\rrvert ^k\bigr),\quad x\in\R_+,a\in\R,b\in\R_+,\
\end{equation}
for all $j=0,1,\dots,l$.
\end{lm}
Now consider $R(t,x)$. Applying Lemma~\ref{lem: g(x,a,b)} with $f'$
instead of $f$ (and thus with $g_1$ instead of $g$), we have
\begingroup
\abovedisplayskip=7pt
\belowdisplayskip=7pt
\begin{align}
\bigl\llvert R(t,x)\bigr\rrvert &\le\frac{1}{2}\sqrt{t}\mathbb{E}\bigl
\llvert \xi g_1\bigl(x,\xi \sqrt{t},Y^{n-1}_t
\bigr)\bigr\rrvert
\nonumber
\\
&\leq Ct\mathbb{E} \bigl[\xi^2 \bigl(1+x^{k_2}+\bigl\llvert
(\xi\sqrt {t})^2+Y^{n-1}_t\bigr
\rrvert^{k_2}\bigr) \bigr]
\nonumber
\\
&\leq Ct\mathbb{E} \bigl[\xi^2 \bigl(1+x^{k_2}+2^{k-1}
\bigl((\xi\sqrt {t})^{2k_2}+\bigl(Y^{n-1}_t
\bigr)^{k_2} \bigr) \bigr) \bigr]
\nonumber
\\
&\leq C\bigl(1+x^{k_2}
\bigr),\quad x\ge0,\ t\in[0,T],\label{eq: estimate R(t,x)}
\end{align}
where the constant $C$ clearly depends only on $C_2$, $k_2$, $T$,
and $n$.

Combining the obtained estimates, we finally get
\begin{align*}
\bigl\llvert \partial_x\mathbb{E}f\bigl(X_t(x)\bigr)
\bigr\rrvert &\leq C\bigl(1+x^{{k}_1}\bigr)+C\bigl(1+x^{k_2}\bigr)
\\
&\leq C\bigl(1+x^k\bigr),\quad x\ge0,\ t\in[0,T],
\end{align*}
where $k=\max\{k_1,k_2\}$, and the constant $C$ depends only on $
C_1, C_2$, $k_1,k_2$, $T$, and $n$.

Now consider the general case where $\delta\ge1$$,\delta\notin
\mathbb{N}$. Note that we consider the general case only for $l=1$
because the reasoning for higher-order derivatives is the same.

Let $n<\delta<n+1 $, $n\in\N$. According to \cite
[Prop.~6.2.1.1]{CirBessel}, $Y^\delta_t(x)$ has the same distribution
as the affine sum of two independent BESQ processes, namely,
\begin{align*}
Y^\delta_t(x)&\stackrel{d} {=}\lambda\widetilde{Y}^n_t(x)+
\lambda _2\widehat{Y}^{n+1}_t(x),
\end{align*}
where
$\widetilde{Y}^n_t(x)$ and $\widehat{Y}^{n+1}_t(x)$ are two
independent BESQ processes of dimensions $n$ and $n+1$,
respectively, starting at $x$, and $\lambda_1=n+1-\delta\in(0,1)$,
$\lambda_2=1-\lambda_1=\delta-n\in(0,1)$ (so that $\delta=\lambda_1
n+\lambda_2(n+1)$). Using the estimates just obtained for $\delta\in\N
$, we have
\begin{align*}
\partial_x\mathbb{E}f\bigl(Y^\delta_t(x)\bigr)
&=\partial_x\mathbb{E}f \bigl(\lambda_1\widetilde
Y^n_t(x)+\lambda _2
\widehat{Y}^{n+1}_t(x) \bigr)
\\
&= \frac{1}{2} \bigl[\partial_x\mathbb{E}f \bigl(
\lambda_1\bigl(\widetilde {Y}^+_t(x)+\widetilde{Y}^{n-1}_t
\bigr) +\lambda_2\bigl(\widehat{Y}^+_t(x)+\widehat
{Y}^{n}_t\bigr) \bigr)
\\
&\quad+\partial_x\mathbb{E}f \bigl(\lambda_1\bigl(
\widetilde {Y}^-_t(x)+\widetilde{Y}^{n-1}_t\bigr)
+\lambda_2\bigl(\widehat{Y}^-_t(x)+\widehat
{Y}^{n}_t\bigr) \bigr) \bigr],
\end{align*}
where
\begin{align*}
\widetilde{Y}^+_t(x)&\,{:=}\,(\sqrt{x}+\tilde{\xi}\sqrt{t})^2,
\qquad\! \widetilde{Y}^-_t(x)\,{:=}\,(\sqrt{x}-\tilde{\xi}
\sqrt{t})^2,\qquad\!\widehat {Y}^+_t(x)\,{:=}\,(\sqrt{x}+\hat{\xi}
\sqrt{t})^2,
\\
\widehat{Y}^-_t(x)&:=(\sqrt{x}-\hat{\xi}\sqrt{t})^2,
\qquad\widetilde {Y}^{n-1}_t:=\sum
_{i=1}^{n-1}\bigl(\widetilde{B}^i_t
\bigr)^2,\qquad\widehat {Y}^{n}_t:=\sum
_{i=1}^{n}\bigl(\widehat{B}^i_t
\bigr)^2,
\end{align*}
\endgroup
with independent standard normal variables $\tilde{\xi}$ and $\hat{\xi
}$ and standard Brownian motions $\widetilde B^i$ and $\widehat
{B}^i$. Using again the fact that the distributions of $\tilde Y^\pm
_t(x)$ and $\hat Y^\pm_t(x)$ coincide and proceeding as in~\eqref
{eq:P(t,x),R(t,x)}, we have
\begin{align*}
\partial_x\mathbb{E}f\bigl(Y^\delta_t(x)\bigr)&= \frac{1}{2} \bigl[\partial_x\mathbb{E}f \bigl(\lambda_1 \bigl((\sqrt {x}+\tilde{\xi}\sqrt{t})^2+\widetilde{Y}^{n-1}_t \bigr) +\lambda_2 \bigl((\sqrt{x}+\hat{\xi}\sqrt{t})^2+\widehat{Y}^{n}_t\bigr) \bigr)\\
&\quad +\partial_x\mathbb{E}f \bigl(\lambda_1\bigl((\sqrt{x}-\tilde {\xi}\sqrt{t})^2+\widetilde{Y}^{n-1}_t\bigr) +\lambda_2 \bigl((\sqrt {x}-\hat{\xi}\sqrt{t})^2+\widehat{Y}^{n}_t \bigr) \bigr) \bigr]\\
&= \frac{1}{2} \biggl[\mathbb{E}f' \bigl(\lambda_1 \bigl((\sqrt{x}+\tilde{\xi }\sqrt{t})^2+\widetilde{Y}^{n-1}_t \bigr) \\
&\quad +\lambda_2 \bigl((\sqrt{x}+\hat {\xi}\sqrt{t})^2+\widehat{Y}^{n}_t\bigr) \bigr) \biggl(1+(\lambda_1\tilde\xi+\lambda_2\hat\xi)\sqrt{\frac{t}x} \biggr)\\
&\quad +\mathbb{E}f' \bigl(\lambda_1 \bigl((\sqrt{x}-\tilde{\xi}\sqrt {t})^2+\widetilde{Y}^{n-1}_t\bigr)\\
&\quad +\lambda_2 \bigl((\sqrt{x}-\hat{\xi }\sqrt{t})^2+\widehat{Y}^{n}_t \bigr) \bigr) \biggl(1-(\lambda_1\tilde\xi +\lambda_2\hat\xi)\sqrt{\frac{t}x} \biggr) \biggr]\\
&=\mathbb{E}f'\bigl(Y^\delta_t(x)\bigr) +\frac{\sqrt{t}}2\E \bigl[(\lambda_1\tilde\xi+\lambda_2\hat\xi)\\
&\quad \times g_1 \bigl(x,(\lambda_1\tilde{\xi}+\lambda_2\hat{\xi})\sqrt{t},\lambda_1\lambda_2(\tilde\xi-\hat{\xi})^2t+\lambda_1\widetilde{Y}^{n-1}_t+\lambda _2\widehat{Y}^{n}_t \bigr) \bigr]\\
&=:P_1(t,x)+R_1(t,x).
\end{align*}

Combination of estimates \eqref{eq: sum a_i} and \eqref{eq:EY^p<
C(1+x^p)} leads to the estimate
\begin{align*}
\bigl\llvert P_1(t,x)\bigr\rrvert &=\bigl\llvert\mathbb{E}f'\bigl(Y^\delta_t(x)\bigr)\bigr\rrvert \leq C_1\mathbb{E}\bigl(1+\bigl\llvert Y^\delta_t(x)\bigr\rrvert^{k_1}\bigr)\\
&\leq C_1\mathbb{E} \bigl(1+4^{k_1-1}\bigl\llvert\lambda_1^{k_1} \bigl(\bigl(\widetilde {Y}^+_t(x)\bigr)^{k_1}\\
&\quad +\bigl(\widetilde{Y}^{n-1}_t\bigr)^{k_1} \bigr) +\lambda_2^{k_1} \bigl(\bigl(\widehat{Y}^+_t(x)\bigr)^{k_1}+\bigl(\widehat{Y}^{n}_t\bigr)^{k_1}\bigr)\bigr\rrvert \bigr)\\
&\leq C\bigl(1+x^{k_1}\bigr),\quad x\ge0,\ t\in[0,T],
\end{align*}
where the constant $C$ depends only on $C_1$, $k_1$, $T$, and $
n$. By Lemma~\ref{lem: g(x,a,b)}, similarly to estimate \eqref{eq:
estimate R(t,x)}, we have
\begin{align*}
\bigl\llvert R_1(t,x)\bigr\rrvert &\leq C\mathbb{E} \bigl[(\lambda_1\tilde{\xi}+\lambda_2\hat{\xi })^2\bigl(1+x^{k}\\
&\quad +\bigl\llvert (\lambda_1\tilde{\xi}+
\lambda_2\hat{\xi })^{2}t+\lambda_1
\lambda_2(\tilde\xi-\hat{\xi})^2t+\lambda_1
\widetilde {Y}^{n-1}_t+\lambda_2
\widehat{Y}^{n}_t\bigr\rrvert^k \bigr) \bigr]
\\
&\leq C\bigl(1+x^{k}\bigr),\quad x\ge0,\ t\in[0,T],
\end{align*}
where the constant $C$ depends only on $C_2$, $k_2$, $T$, and $
n$. Combining the last two estimates, we get
\begin{align}
\label{eq: dx Ef(Xt(x))} \bigl\llvert \partial_x\mathbb{E}f
\bigl(X_t(x)\bigr)\bigr\rrvert \leq C\bigl(1+x^k\bigr),
\quad x\ge0,\ t\in[0,T],
\end{align}
where $k=\max\{k_1,k_2\}$, and the constant $C$ depends only on $
C_1, C_2$, $k_1,k_2$, and $T$.

\emph{Step 2}. Let $l=2$. From \emph{Step 1} we have
\begin{align*}
\partial_x\mathbb{E}f\bigl(Y^n_t(x)\bigr)&=
\mathbb{E}f' \bigl(Y^{n}_{t}(x) \bigr)+
\frac{1}{2}\sqrt{t}\,\mathbb{E} \bigl(\xi g_1\bigl(x, \xi\sqrt
t, Y^{n-1}_{t}\bigr) \bigr).
\end{align*}
Therefore,
\begin{align*}
\partial_x^2\mathbb{E}f\bigl(Y^n_t(x)\bigr)&=\partial_x\mathbb{E}f' \bigl(Y^{n}_{t}(x)\bigr)+\frac{1}{2}\sqrt{t}\,\mathbb{E} \bigl(\xi\partial_x g_1\bigl(x, \xi\sqrt t, Y^{n-1}_{t}\bigr)\bigr)\\
&=:P_2(t,x)+R_2(t,x).
\end{align*}
From estimate \eqref{eq: dx Ef(Xt(x))} with $f$ replaced by $f'$ we obtain
\begin{align}
\label{eq: P2(t,x)} \bigl\llvert P_2(t,x)\bigr\rrvert &\leq C
\bigl(1+x^{k_3}\bigr),\quad x\ge0,\ t\in[0,T],
\end{align}
where the constant $C$ depends only on $C_1$, $C_3$, $k_1$, $
k_3$, $T$, and $n$. For $R_2(t,x)$, applying Lemma~\ref{lem:
g(x,a,b)} once more to~$g_1$ instead of $g$, we get
\begin{align*}
\bigl\llvert R_2(t,x)\bigr\rrvert &\,{\leq}\,\frac{1}{2}\sqrt{t}\mathbb{E}\bigl\llvert { \xi\partial _xg_1\bigl(x, \xi\sqrt t, Y^{n-1}_{t}\bigr)}\bigr\rrvert \,{\leq}\,Ct\mathbb{E}\bigl(\xi ^2\bigl(1+x^{k}+\llvert \xi\sqrt{t}\rrvert^{k}+\bigl(Y^{n-1}_t\bigr)^{k}\bigr) \bigr)\\
&\leq C\bigl(1+x^{k}\bigr),\quad x\ge0,\ t\in[0,T],
\end{align*}
where the constants $C$ and $k\in\N$ depend only on $\{(C_i,k_i)$,
$i=1,2,3,4\}$, $T$, and $n$.
Combining the obtained estimates, we finally get
\begin{align*}
\bigl\llvert \partial_x^2\,\mathbb{E}f
\bigl(X_t(x)\bigr)\bigr\rrvert &\leq C\bigl(1+x^k\bigr),
\quad x\ge 0,\ t\in[0,T],
\end{align*}
where the constants $C$ and $k\in\N$ depend only on $\{(C_i,k_i)$,
$i=1,2,3,4$\}, $T$, and $n$.

\emph{Step 3}. Now we may continue by induction on $l$. Suppose that
estimate \eqref{eq: nel. d/dx Ef(Xt(x)} is valid for $l=m-1$. Let us
show that it is still valid for $l=m$. The arguments are similar to
those in the case $m=2$ (Step 2). We have
\begin{align*}
\partial_x^m\,\mathbb{E}f\bigl(Y^n_t(x)
\bigr)&=\partial_x^{m-1}\mathbb{E}f'
\bigl(Y^{n}_{t}(x) \bigr) +\frac{1}{2}\sqrt{t}\,
\partial_x^{m-1}\mathbb{E} \bigl(\xi g_1\bigl(x,
\xi \sqrt t, Y^{n-1}_{t}\bigr) \bigr)\\
&=:P_m(t,x)+R_m(t,x).
\end{align*}
Then, similarly to estimates \eqref{eq: P(t,x)<...} and \eqref{eq:
P2(t,x)}, we have
\begin{align*}
\bigl\llvert P_m(t,x)\bigr\rrvert &\leq C\bigl(1+x^{k_{m}}
\bigr),
\end{align*}
where the constant $C$ depends only on $\{(C_i,k_i)$, $i=1,3,\dots
, 2m-1\}$, $T$, and $n$.

For $R_m(t,x)$, applying Lemma~\ref{lem: g(x,a,b)} to $g_1$ instead of
$g$, we get
\begin{align*}
\bigl\llvert R_m(t,x)\bigr\rrvert &\leq\frac{1}{2}\sqrt{t}
\mathbb{E}\bigl\llvert \xi\partial_x^{m-1} g_1
\bigl(x, \xi\sqrt t, Y^{n-1}_{t}\bigr)\bigr\rrvert \leq C
\bigl(1+x^{k}\bigr),
\end{align*}
where the constants $C$ and $k\in\N$ depend only on $\{(C_i,k_i)$,
$i=1,\dots,2m\}$, $T$, and~$n$. Combining the obtained
estimates, we get
\begin{align*}
\bigl\llvert \partial_x^{m}\mathbb{E}f
\bigl(X_t(x)\bigr)\bigr\rrvert &\leq C\bigl(1+x^k\bigr),
\quad x>0,\ t\in[0,T],
\end{align*}
where the constants $C$ and $k\in\N$ depend only on $\{(C_i,k_i)$,
$i=1,\dots,2m\}$, $T$, and~$n$. Thus, Theorem~\ref{pol1} is
proved for all $l\in\mathbb{N}$.

\emph{Step 4}. As in Alfonsi \cite[p.~28]{Alfonsi}, inequality~\eqref
{eq: nel. d/dx Ef(Xt(x)} for the derivatives with respect to $t$ and
mixed derivatives follows automatically by an induction on $l'$ using
that, for $l'\ge1$ such that $4l'+2l\le q$,
\begin{align*}
\partial^l_x\partial^{l'}_tu(t,x)
&=\partial^l_x \biggl(\theta(\kappa-x)
\partial_x\partial^{l'-1}_tu(t,x) +
\frac{\sigma^2}{2}x\partial^2_x\partial^{l'-1}_tu(t,x)
\biggr)
\\
&=\frac{\sigma^2}{2}x\partial^{l+2}_x\partial^{l'-1}_tu(t,x)
+ \biggl(l\frac{\sigma^2}{2}+\theta(\kappa-x) \biggr)\partial^{l+1}_x
\partial ^{l'-1}_tu(t,x)\\
&\quad  -l\theta\partial^{l}_x\partial^{l'-1}_tu(t,x).\qedhere
\end{align*}
\end{proof}

\begin{appendix}

\section{Appendix: Proof of Lemma~\ref{lem: g(x,a,b)}}
\label{g derivative}
\begin{proof}
First, let $n=5$ ($l=2$), that is, $f\in C^{5}(\R_+)$. Then,
denoting $A:=a^2+b$, for $i=0,\dots,4$, we have
\begingroup
\begin{align*}
g_i(x;a,b)&=\frac{f^{(i)}(A+x+2a\sqrt x)-f^{(i)}(A+x-2a\sqrt x)}{\sqrt
{x}}
\\
&=\frac{1}{\sqrt{x}}f^{(i)}(A+x+2a\sqrt xs)\bigg|_{s=-1}^{s=1}\\
&=\frac{1}{\sqrt{x}}\int_{-1}^1f^{(i+1)}(A+x+2a
\sqrt xs)2a\sqrt x\,ds
\\
&=2a\int_{-1}^1f^{(i+1)}(A+x+2a\sqrt xs)
\,ds,\quad x>0.
\end{align*}
From this it follows that
\[
\lim\limits
_{x\downarrow0}g_i(x;a,b)=2a\int_{-1}^1f^{(i+1)}(A)
\, ds=4af^{(i+1)}(A).
\]
In particular, every function $g_i$, $i=0,\dots,4$, is continuously
extendable to the whole half-line $\R_+=[0,\infty)$ by defining $
g_i(0;a,b):=4af^{(i+1)}(A)$.

Let, moreover, $f\in C^{5}_{\mathit{pol}}(\R_+)$ with the estimates
\begin{equation}
\label{eq:fi(x) pol growth constants} \bigl\llvert f^{(i)}(x)\bigr\rrvert \le C_i
\bigl(1+x^{k_i}\bigr),\quad x\ge0,\ i=0,1,\ldots,5,
\end{equation}
for some constants $C_i>0$ and $k_i\in\N$, $i=0,1,\ldots, 5$.

Then we have the estimate
\begin{align*}
\bigl\llvert g_i(x;a,b)\bigr\rrvert &\le2\llvert a\rrvert \int_{-1}^1\bigl\llvert f^{(i+1)}(A+x+2a\sqrt xs)\bigr\rrvert \,ds \\
&\le4\llvert a\rrvert C_{i+1} \bigl(1+\bigl(A+x+2 \llvert a\rrvert \sqrt x\bigr)^{k_{i+1}} \bigr)\\
&\le4\llvert a\rrvert C_{i+1} \bigl(1+\bigl(A+a^2+2x \bigr)^{k_{i+1}} \bigr) \\
&\le C\llvert a\rrvert \bigl(1+A^{k_{i+1}}+x^{k_{i+1}}\bigr),\quad x\ge0,
\end{align*}
\endgroup
where $C$ depends on $C_{i+1}$ and $k_{i+1}$ only.\vadjust{\eject}

Now let us concentrate ourselves on the derivatives of $g=g_0$ with
respect to $x$. We have
\begin{align}
g'_0(x;a,b)&=2a\int_{-1}^1f''(A+x+2a\sqrt xs)\, \biggl(1+\frac{as}{\sqrt x} \biggr)ds\nonumber\\
&=2a\int_{-1}^1f''(A+x+2a\sqrt xs)\,ds +\frac{2a^2}{\sqrt x}\int_{-1}^1f''(A+x+2a\sqrt xs)\,sds\nonumber\\
&=g_1(x;a,b)+\frac{2a^2}{\sqrt x}\int_{-1}^1\Biggl(f''(A+x)\nonumber\\
&\quad +\int_0^sf'''(A+x+2a\sqrt xu)\,2a\sqrt x\,du \Biggr)s\,ds\nonumber\\
&=g_1(x;a,b)+\frac{2a^2f''(A+x)}{\sqrt x}\int_{-1}^1s\,ds\nonumber\\
&\quad  +4a^3\int_{-1}^1\int_0^sf'''(A+x+2a\sqrt xu)\,\,du\,s\,ds\nonumber\\
&=g_1(x;a,b)+4a^3\int_{-1}^1\int_0^sf'''(A+x+2a\sqrt xu)\,\,du\,s\, ds,\quad x>0.\label{eq: dg_0/dx}
\end{align}

(Note that the term at the negative power of $x$, that is, at $
1/\sqrt x$, vanishes since
$\int_{-1}^1s\,ds=0$.)
From this it follows that there exists the limit
\begin{align*}
\lim_{x\downarrow0}g'_0(x;a,b) &=\lim
_{x\downarrow0} g_1(x;a,b) +4a^3\int
_{-1}^1 \int_{0}^sf'''(A)
\,du\,s\,ds
\\
&=4af''(A)+\frac{8a^3}{3}f'''(A).
\end{align*}
In particular, the function $g=g_0$ is continuously differentiable at
$x=0$ and thus belongs to $C^1(\R_+)$ since $g_0'(0;a,b)=\lim_{x\downarrow0}g'_0(x;a,b)$ by the Lagrange theorem.

If, moreover, $f\in C^{5}_{\mathit{pol}}(\R_+)$ satisfies estimates
\eqref{eq: polinom. f isv. augimas iki n}\ for $i\le5$, then we have
the corresponding estimate
for $g'_0$:
\begin{align}
\bigl\llvert g'_0(x;a,b)\bigr\rrvert &\le\bigl\llvert
g_1(x;a,b)\bigr\rrvert +4\llvert a\rrvert^3\int
_{-1}^1 \int_{-s}^s
\bigl\llvert f'''(A+x+2a\sqrt xu)\bigr
\rrvert \, du\,\llvert s\rrvert \,ds
\nonumber
\\
&\le C_2\llvert a\rrvert \bigl(1+A^{k_{2}}+x^{k_{2}}
\bigr) \nonumber\\
&\quad +4\llvert a\rrvert^3\int_{0}^1
\int_{-s}^sC_3\bigl(1+\bigl(A+x+2
\llvert a\rrvert \sqrt xu\bigr)^{k_3}\bigr)\,du\, ds
\nonumber
\\
&\le C_2\llvert a\rrvert \bigl(1+A^{k_{2}}+x^{k_{2}}
\bigr)+4\llvert a\rrvert ^3C_3\bigl(1+A^{k_{3}}+x^{k_{3}}
\bigr)
\nonumber
\\
&\le C\llvert a\rrvert \bigl(1+A^{k}+x^{k} \bigr),\quad x
\ge\xch{0,}{0.}\label{eq:g'0 pol. growth}
\end{align}
where $C$ and $k$ depend on $C_{2,3}$, $k_{2,3}$, and $
A=a^2+b$ only.

Thus, we have proved that $g=g_0\in C_{\mathit{pol}}^1(\R_+)$, provided that $
f\in C_{\mathit{pol}}^5(\R_+).$ (In fact, for estimate \eqref{eq:g'0 pol.
growth}, it suffices that $f\in C_{\mathit{pol}}^3(\R_+).$) More precisely, if
\[
\bigl\llvert f^{(i)}(x)\bigr\rrvert \le C_i
\bigl(1+x^{k_i}\bigr),\quad x\ge0,\ i=1,2,3,
\]
then
\[
\bigl\llvert g^{(j)}_0(x;a,b)\bigr\rrvert \le C
\bigl(1+A^k+x^k\bigr),\quad x\ge0,\ j=0,1,
\]
where the constants $C>0$ and $k\in\N$ depend only on $C_i$ and $
k_i$, $i=1,2,3$, and, in particular, on a good set of the function $
f\in C^{5}_{\mathit{pol}}(\R_+)$.

Now, let us proceed to the second derivative of $g_0$. From Eq.~\eqref
{eq: dg_0/dx} we have
\begin{align*}
g_0''(x;a,b)&=g_1'(x;a,b)+4a^3\int_{-1}^1\int_0^sf^{(4)}(A+x+2a\sqrt xu) \biggl(1+\frac{au}{\sqrt x} \biggr)\,du\,s\,ds\\
&= \Biggl(2a\int_{-1}^1f^{(2)}(A+x+2a\sqrt{x}s)\,ds \Biggr)'\\
&\quad +4a^3\int_{-1}^1\int_0^sf^{(4)}(A+x+2a\sqrt xu)\,du\,s\,ds\\
&\quad+\frac{4a^4}{\sqrt{x}}\int_{-1}^1\int_0^sf^{(4)}(A+x+2a\sqrt xu)u\,du\,s\,ds\\
&=2a\int_{-1}^1 f^{(3)}(A+x+2a\sqrt{x}s) \biggl(1+\frac{as}{\sqrt{x}} \biggr)\,ds\\
&\quad +4a^3\int_{-1}^1\int_0^sf^{(4)}(A+x+2a\sqrt xu)\,du\,s\, ds\\
&\quad+\frac{4a^4}{\sqrt{x}}\int_{-1}^1\int_0^s \Biggl[f^{(4)}(A+x)\\
&\quad +\int_0^uf^{(5)}(A+x+2a\sqrt xv)2a\sqrt{x}\,dv \Biggr]u\,du\, s\,ds\\
&=2a\int_{-1}^1 f^{(3)}(A+x+2a\sqrt{x}s)\,ds\\
&\quad +\frac{2a^2}{\sqrt{x}}\int_{-1}^1\Biggl[f^{(3)}(A+x)+\int_0^sf^{(4)}(A+x+2a\sqrt xu)2a\sqrt {x}\,du \Biggr]\,s\,ds\\
&\quad+4a^3\int_{-1}^1\int_0^sf^{(4)}(A+x+2a\sqrt xu)\,du\,s\,ds\\
&\quad +\frac{4a^4}{\sqrt{x}}f^{(4)}(A)\int_{-1}^1\int_0^s u\,du\,s\,ds\\
&\quad+8a^5\int_{-1}^1\int_0^s\int_0^uf^{(5)}(A+x+2a\sqrt xv)\, dv\,u\,du\,s\,ds\\
&=2a\int_{-1}^1 f^{(3)}(A+x+2a\sqrt{x}s)\,ds\\
&\quad +4a^3\int_{-1}^1\int_0^sf^{(4)}(A+x+2a\sqrt xu)\,du\,s\,ds\\
&\quad+4a^3\int_{-1}^1\int_0^sf^{(4)}(A+x+2a\sqrt xu)\,du\,s\,ds\\
&\quad+8a^5\int_{-1}^1\int_0^s\int_0^uf^{(5)}(A+x+2a\sqrt xv)\, dv\,u\,du\,s\,ds,\quad x>0.
\end{align*}
(Note that, again, the term at the negative power of $x$, that is, at
$1/\sqrt x$, vanishes since $\int_{-1}^1\int_0^s u\,du\,s\,ds=0$.)
In~particular, again by the Lagrange theorem, $g_0$ is twice
continuously differentiable on the whole half-line $\R_+$ since there
exists the finite limit
\begin{align*}
\lim_{x\downarrow0}g''_0(x;a,b)&=
\lim_{x\downarrow
0}g_1'(x;a,b)+4a^3f^{(4)}(A)
\int_{-1}^1\int_0^s
\,du\,s\,ds
\\
&\qquad+8a^5f^{(5)}(A)\int_{-1}^1
\int_0^s\int_0^u
\,dv\,u\,du\,s\,ds
\\
&=4af^{(3)}(A)+\frac{16a^3f^{(4)}(A)}{3}+\frac{16a^5f^{(5)}(A)}{15}.
\end{align*}
If, moreover, $f\in C^{5}_{\mathit{pol}}(\R_+)$ satisfies estimates
\eqref{eq:fi(x) pol growth constants}, then we have the corresponding estimate
for~$g''_0$:
\begin{align}
\bigl\llvert g_0''(x,a,b)\bigr\rrvert & \leq2\llvert a\rrvert \int_{-1}^1\bigl\llvert f^{(3)}(A+x+2a\sqrt{x}s)\bigr\rrvert \, ds\nonumber\\
&\quad +8\llvert a\rrvert^3\int_{-1}^1\int_{-s}^s\bigl\llvert f^{(4)}(A+x+2a\sqrt{x}u)\bigr\rrvert \,du\,\llvert s\rrvert \, ds\nonumber\\
&\quad +8\llvert a\rrvert ^5\int_{-1}^1\int_{-s}^s\int_{-u}^u\bigl\llvert f^{(5)}(A+x+2a\sqrt {x}v)\bigr\rrvert \,dv\,\llvert u\rrvert\,du\,\llvert s\rrvert \,ds\nonumber\\
&\leq2\llvert a\rrvert \int_{-1}^1C_3\bigl(1+\bigl(A+x+2\llvert a\rrvert \sqrt{x}s\bigr)^{k_3}\bigr)\,ds\nonumber\\
&\quad +8\llvert a\rrvert ^3\int_{0}^1\int_{-s}^sC_4\bigl(1+\bigl(A+x+2\llvert a\rrvert \sqrt{x}u\bigr)^{k_4}\bigr)\,du\,ds\nonumber\\
&\quad +8\llvert a\rrvert ^5\int_{0}^1\int_{-s}^s\int_{-u}^uC_5\bigl(1+\bigl(A+x+2\llvert a\rrvert \sqrt {x}v\bigr)^{k_5}\bigr)\,dv\,du\,ds\nonumber\\
&\leq C\llvert a\rrvert \bigl(1+A^k+x^k\bigr),\quad x \ge0,\label{eq:g''0 pol. growth}
\end{align}
where the constants $C>0$ and $k\in\N$ depend only on $C_i$ and $
k_i$, $i=3$,4,5, and, in particular, on a good set of the function $
f\in C^{5}_{\mathit{pol}}(\R_+)$.

Now, for $l>2$, we can proceed similarly. For $f\in C^{2l+1}_{\mathit{pol}}(\R
_+)$, denote
\begin{align*}
F^{0,q}&=F^{0,q}(x,a,b):=\int_{-1}^1f^{(q)}(A+x+2a\sqrt{x}s)\,ds,\\
F^{p,q}&=F^{p,q}(x,a,b)\\
&:=\int_{-1}^1\int_{0}^{s_1}\dots\int_{0}^{s_p}f^{(q)}(A+x+2a\sqrt{x}s_{p+1})\,ds_{p+1}\ldots\,s_2\,ds_2\, s_1\,ds_1,
\\
&\qquad p=1,\dots,l,\,q=l+1,\ldots,2l+1.
\end{align*}
\newcommand{\dotrule}[1]{%
\parbox[t]{#1}{\dotfill}}
Then, in addition to the first two derivatives
\begin{align}
&g'_0(x,a,b)=a\bigl(2F^{0,2}+4a^2F^{1,3}\bigr)\quad\mbox{and}\nonumber\\
&g''_0(x,a,b)=a\bigl(2F^{0,3}+8a^2F^{1,4}+8a^4F^{2,5}\bigr),\qquad\qquad\qquad\nonumber
\end{align}
we get:
\begin{align}
&g'''_0(x,a,b)=a\bigl(2F^{0,4}+12a^2F^{1,5}+24a^4F^{2,6}+16a^6F^{3,7}\bigr),\nonumber\\
&\ldots\ldots\ldots\ldots\ldots\ldots\ldots\ldots\ldots\ldots\nonumber\\
&g_0^{(l)}(x,a,b)=a\sum_{j=0}^lc_{j,l}a^{2j}F^{j,l+j+1}(x;a,b), \label{eq:g, lth derivative}
\end{align}
where $c_{j,l}$, $0\le j\le l$, are some constants. Note that, as
before, in the right-hand side of Eq.~\eqref{eq:g, lth derivative},
there are no negative powers of $x$, so that $g_0$ is $l$ times
continuously differentiable on the whole half-line $\R_+$, provided
that $f\in C^{2l+1}_{\mathit{pol}}(\R_+)$.
Moreover, as before, from \eqref{eq:g, lth derivative} we get the
following estimates for $g_0^{(r)}$:
\begin{align*}
\bigl\llvert g_0^{(r)}(x,a,b)\bigr\rrvert \leq C\llvert a
\rrvert \bigl(1+A^k+x^k\bigr), \quad x\geq0,\ r=0,1,
\dots,l,
\end{align*}
where the constants $C>0$ and $k\in\N$ depend only on $C_i$ and $
k_i$, $i=0,\dots,2l+1$, that is, only on a good set of the function
$f\in C^{2l+1}_{\mathit{pol}}(\R_+)$.
\end{proof}
\end{appendix}

\begin{acknowledgement}
We thank the anonymous reviewer for comments that enabled us to improve
the presentation of the paper.
\end{acknowledgement}


\end{document}